\documentclass[reqno,11pt]{amsart}

\usepackage{amsmath,amsfonts,amssymb,amsthm,graphicx}
\usepackage{enumerate}

 \allowdisplaybreaks
\numberwithin{equation}{section}

\font\script=rsfs10 at 11pt
\font\smallscript=rsfs10 at 8pt
\def\eps{\varepsilon}
\def\Q{\mathcal Q}
\def\H{{\mbox{\script H}\,\,}}

\def\comp{\subset\subset}
\def\C{{\mathcal C}}
\def\M{{\mathcal M}}
\def\R{{\mathbb R}}
\def\D{{\mathcal D}}
\def\G{{\mbox{\script G}\,}}
\def\smG{{\mbox{\smallscript G}\,}}
\def\QQ{{\mathcal Q}}
\def\N{\mathbb N}
\def\NN{\mathcal N}
\def\bal{\begin{aligned}}
\def\eal{\end{aligned}}
\def\res{\mathop{\hbox{\vrule height 7pt width .5pt depth 0pt \vrule height .5pt width 6pt depth 0pt}}\nolimits}
\def\proofof#1{\begin{proof}[Proof of #1]}
\def\step#1#2{\par\noindent{\underline{\it Step~#1.}}\emph{ #2}\\}
\def\claim#1#2{\par\noindent{\underline{\bf Claim~#1.}}\emph{ #2}\\}

\def\S{\mathbb S}
\def\mathS{\mathcal S}

\def\u#1{\hbox{\boldmath $#1$}} 

\def\XXint#1#2#3{{\setbox0=\hbox{$#1{#2#3}{\int}$} \vcenter{\vspace{-1pt}\hbox{$#2#3$}}\kern-.5\wd0}}

\theoremstyle{plain}
\newtheorem{lemma}{Lemma}[section]

\newtheorem{theorem}[lemma]{Theorem}

\newtheorem{defin}[lemma]{Definition}
\newtheorem{remark}[lemma]{Remark}

\def\de0#1{\rule[3pt]{#1}{0.4pt} \hspace{-0.1pt} \rule[3.05pt]{0.05pt}{0.4pt} \hspace{-0.1pt} \rule[3.1pt]{0.05pt}{0.4pt} \hspace{-0.1pt} \rule[3.15pt]{0.05pt}{0.4pt} \hspace{-0.1pt} \rule[3.2pt]{0.05pt}{0.4pt} \hspace{-0.1pt} \rule[3.25pt]{0.05pt}{0.4pt} \hspace{-0.1pt} \rule[3.3pt]{0.05pt}{0.4pt} \hspace{-0.1pt} \rule[3.35pt]{0.05pt}{0.4pt} \hspace{-0.1pt} \rule[3.4pt]{0.05pt}{0.4pt} \hspace{-0.1pt} \rule[3.45pt]{0.05pt}{0.4pt} \hspace{-0.1pt} \rule[3.5pt]{0.05pt}{0.4pt} \hspace{-0.1pt} \rule[3.55pt]{0.05pt}{0.4pt} \hspace{-0.1pt} \rule[3.6pt]{0.05pt}{0.4pt} \hspace{-0.1pt} \rule[3.65pt]{0.05pt}{0.4pt} \hspace{-0.1pt} \rule[3.7pt]{0.05pt}{0.4pt} \hspace{-0.1pt} \rule[3.75pt]{0.05pt}{0.4pt} \hspace{-0.1pt} \rule[3.8pt]{0.05pt}{0.4pt} \hspace{-0.1pt} \rule[3.85pt]{0.05pt}{0.4pt} \hspace{-0.1pt} \rule[3.9pt]{0.05pt}{0.4pt} \hspace{-0.1pt} \rule[3.95pt]{0.05pt}{0.4pt} \hspace{-0.1pt} \rule[4.0pt]{0.05pt}{0.4pt} \hspace{-0.1pt} \rule[4.05pt]{0.05pt}{0.4pt} \hspace{-0.1pt} \rule[4.1pt]{0.05pt}{0.4pt} \hspace{-0.1pt} \rule[4.15pt]{0.05pt}{0.4pt} \hspace{-0.1pt} \rule[4.2pt]{0.05pt}{0.4pt} \hspace{-0.1pt} \rule[4.25pt]{0.05pt}{0.4pt} \hspace{-0.1pt} \rule[4.3pt]{0.05pt}{0.4pt} \hspace{-0.1pt} \rule[4.35pt]{0.05pt}{0.4pt} \hspace{-0.1pt} \rule[4.4pt]{0.05pt}{0.4pt} \hspace{-0.1pt} \rule[4.45pt]{0.05pt}{0.4pt} \hspace{-0.1pt} \rule[4.5pt]{0.05pt}{0.4pt} \hspace{-0.1pt} \rule[4.55pt]{0.05pt}{0.4pt} \hspace{-0.1pt} \rule[4.6pt]{0.05pt}{0.4pt} \hspace{-0.1pt} \rule[4.65pt]{0.05pt}{0.4pt} \hspace{-0.1pt} \rule[4.7pt]{0.05pt}{0.4pt} \hspace{-0.1pt} \rule[4.75pt]{0.05pt}{0.4pt} \hspace{-0.1pt} \rule[4.8pt]{0.05pt}{0.4pt} \hspace{-0.1pt} \rule[4.85pt]{0.05pt}{0.4pt} \hspace{-0.1pt} \rule[4.9pt]{0.05pt}{0.4pt} \hspace{-0.1pt} \rule[4.95pt]{0.05pt}{0.4pt} \hspace{-0.1pt} \rule[5.0pt]{0.05pt}{0.4pt} \hspace{-0.1pt} \rule[5.05pt]{0.05pt}{0.4pt} \hspace{-0.1pt} \rule[5.1pt]{0.05pt}{0.4pt} \hspace{-0.1pt} \rule[5.15pt]{0.05pt}{0.4pt} \hspace{-0.1pt} \rule[5.2pt]{0.05pt}{0.4pt} \hspace{-0.1pt} \rule[5.25pt]{0.05pt}{0.4pt} \hspace{-0.1pt} \rule[5.3pt]{0.05pt}{0.4pt} \hspace{-0.1pt} \rule[5.35pt]{0.05pt}{0.4pt} \hspace{-0.1pt} \rule[5.4pt]{0.05pt}{0.4pt} \hspace{-0.1pt} \rule[5.45pt]{0.05pt}{0.4pt} \hspace{-0.1pt} \rule[5.5pt]{0.05pt}{0.4pt} \hspace{-0.1pt} \rule[5.55pt]{0.05pt}{0.4pt} \hspace{-0.1pt} \rule[5.6pt]{0.05pt}{0.4pt} \hspace{-0.1pt} \rule[5.65pt]{0.05pt}{0.4pt} \hspace{-0.1pt} \rule[5.7pt]{0.05pt}{0.4pt} \hspace{-0.1pt} \rule[5.75pt]{0.05pt}{0.4pt} \hspace{-0.1pt} \rule[5.8pt]{0.05pt}{0.4pt} \hspace{-0.1pt} \rule[5.85pt]{0.05pt}{0.4pt} \hspace{-0.1pt} \rule[5.9pt]{0.05pt}{0.4pt} \hspace{-0.1pt} \rule[5.95pt]{0.05pt}{0.4pt} \hspace{-0.1pt} \rule[6.0pt]{0.05pt}{0.4pt}}	

\begin{document}

\title{Comparison between the non-crossing and the non-crossing on lines properties}

\author{D. Campbell}
\author{A. Pratelli}
\author{E. Radici}

\begin{abstract}
In the recent paper~\cite{DPP}, it was proved that the closure of the planar diffeomorphisms in the Sobolev norm consists of the functions which are non-crossing (NC), i.e., the functions which can be uniformly approximated by continuous one-to-one functions on the grids. A deep simplification of this property is to consider curves instead of grids, so considering functions which are non-crossing on lines (NCL). Since the NCL property is way easier to check, it would be extremely positive if they actually coincide, while it is only obvious that NC implies NCL. We show that in general NCL does not imply NC, but the implication becomes true with the additional assumption that $\det(Du)>0$ a.e.\,, which is a very common assumption in nonlinear elasticity.
\end{abstract}

\maketitle

\section{Introduction}

In the framework of nonlinear elasticity, it is important to consider functions which are Sobolev limits of diffeomorphisms, since all these functions are meaningful deformations. Notice that such planar limits, at least for $p<2$, are not necessarily continuous, and they can fail both to be injective and to be surjective. In fact, some area can be shrunken to a point, or vice versa a point can be stretched to a positive area (these are the so-called \emph{cavitations}, around which the function behaves like $x\mapsto x/|x|$). On the other hand, it is clear that not every Sobolev function can be the limit of diffeomorphisms. In particular, this is not possible for a function which is ``very far'' from being invertible. In order to give a precise description of this kind of functions, M\"uller and Spector (\cite{MS}, see also the earlier work of \v Sver\'ak~\cite{Sve}) in the 1990's introduced the notion of \emph{INV functions}, see Definition~\ref{defINV}. Roughly speaking, a function satisfies the INV property if, for any ball $B$ in the domain, points which are inside (resp., outside) $B$ are sent inside (resp., outside) the image of $\partial B$. Since of course any diffeomorphism enjoys the INV property, and this property is preserved under Sobolev limits, then every limit of diffeomorphisms is an INV function. A natural conjecture would have been that INV functions in fact coincide with the closure of diffeomorphisms, but it has recently been shown that it is not so, and this closure is made by the \emph{non-crossing (NC) functions}, see Definition~\ref{defNC}. Roughly speaking, a function $u$ is said to be non-crossing if for any one-dimensional ``grid'' $\G$ inside the domain it is possible to find a continuous, \emph{injective} function on $\G$ which is arbitrarily close in the uniform sense to $u$.

Only for simplicity of notations, we restrict our attention to functions which map the unit square $\Q=[-1,1]^2$ on the unit square $\u\Q=[-1,1]^2$, and which coincide with the identity on the boundary. The characterisation of the Sobolev closure of diffeomorphims is then the following (see~\cite[Theorem~A and~B]{DPP}).

\begin{theorem}\label{thDPP} Let $\D$ be the family of the diffeomorphisms of $\Q$ onto $\u\Q$ which coincide with the identity on the boundary. For any $1\leq p<+\infty$, a $W^{1,p}$ function $u:\Q\to\u\Q$ is a weak $W^{1,p}$ limit of a sequence in $\D$ if and only if it is a strong $W^{1,p}$ limit of a sequence in $\D$, and if and only if it is a non-crossing function.
\end{theorem}

A drawback of this characterisation is that the non-crossing property is not simple to check in particular around the vertices of a grid. The property becomes extremely simpler to check if one only considers non self-intersecting curves, so completely avoiding the troubles that may occur around vertices. In doing so, one finds a weaker property, that we call \emph{non-crossing on lines (NCL)}, see Definition~\ref{defNCL}.

Since this property is much simpler, it would be very nice if it actually coincides with the NC property. The goal of this paper is precisely to study whether or not this is true, and we are able to give a precise answer. In Section~\ref{sec:ctx} we will show that this is not true in general, in fact the NCL property does not even imply the weaker INV property. On the bright side, in Section~\ref{sec:bright} we show that this becomes true for functions $u$ such that $\det(Du)>0$ a.e.\,. This is a very common property in the non-linear elasticity, which is often taken for granted. Indeed, a function $u$ for which this property is false maps a set of positive measure onto a null set, thus some mass ``disappears'', and this is often not admitted. In addition, such a function has an infinite elastic energy under most models.

\section{Definition of NC, NCL and INV maps}

This section is devoted to present the definition of the NC, the NCL, and the INV maps, and to set some notation.\par

As usual, for every $x\in\R^2$ and $r>0$ we denote by $B(x,r)$ the disk centered at $x$ and with radius $r$, and for brevity we write $B_r$ in place of $B(0,r)$. For every function $u\in W^{1,1}(\QQ;\u\QQ)$, we call $\NN_u$ the set of the discontinuity points of $u$. For a generic Sobolev function $u$, $\NN_u$ is a Borel, negligible set. If $u$ is an INV function, the set $\NN_u$ is actually $\H^1$-negligible (see~\cite[Lemma~2.5]{DPP}).

\begin{defin}[INV functions]\label{defINV}
Let $u:\Q\to\u\Q$ be a $W^{1,1}$ function. We say that \emph{$u$ is an INV function} if the following holds. Let $B(x,r)\comp \Q$ be any ball such that the restriction of $u$ to the circle $\mathS=\partial B(x,r)$ is $W^{1,1}$, so in particular $u(\mathS)$ is a continuous, closed curve, possibly self-crossing. Let moreover $z\in\Q\setminus \NN_u$ be a point such that $u(z)\notin u(\mathS)$, so in particular $z\notin\mathS$. Then, if $z\notin B(x,r)$ the point $u(z)$ has degree $0$ with respect to the curve $u(\mathS)$, while if $z\in B(x,r)$ the point $u(z)$ has non-zero degree with respect to the same curve.
\end{defin}

The above definition may seem quite technical at first glance, but in fact it is very reasonable. It is basically saying that every point which is contained ``inside'' a curve has an image which is ``inside'' the image of the curve, and every point which is outside has image outside the image of the curve, in the sense of degree. In other words, ``what is inside remains inside, and what is outside remains outside''.

\begin{defin}[NC functions]\label{defNC}
Let $u:\QQ\to\u\QQ$ be a Sobolev map for which the set $\NN_u$ of the discontinuity points of $u$ is $\H^1$-negligible and which equals the identity on $\partial\Q$. We say that \emph{$u$ is a non-crossing (NC) function} if the following holds. Let $\Gamma\subseteq \Q\setminus\NN_u$ be a finite union of injective, Lipschitz curves such that any two curves have either empty intersection or a single intersection point, and in this case both curves have a tangent vector at the intersection point, and the two vectors are not parallel, and moreover all these intersection points between the curves are different. Then, for every $\eps>0$ there exists a continuous and injective function $v:\Gamma\to\u\Q$, coinciding with the identity on $\partial\Q\cap\Gamma$, such that $\|v-u\|_{L^\infty(\Gamma)}<\eps$.
\end{defin}

Also the definition of the NC functions is rather clear. One does not ask the function to be injective, but to be uniformly close to an injective function on every ``grid'' $\Gamma$. It is also simple to realize the reason for the name. Namely, the simplest possibility for a function not to fulfill the definition, is that two disjoint curves on $\Q$ have images which are two crossing curves on $\u\Q$. Notice that, if this happens, then the property of Definition~\ref{defNC} already fails for a set $\Gamma$ made by a single curve. In turn, of course the above property is radically simpler to state and to check if one restricts himself to the case of a single, Lipschitz, injective curve, instead of finitely many ones which might intersect each other. Also because, clearly, the vertices of the grid (i.e., the intersection points between the different curves) are the most delicate points to treat when checking the validity of the NC property. As a consequence, the following definition is quite natural.

\begin{defin}[NCL functions]\label{defNCL}
Let $u:\QQ\to\u\QQ$ be a Sobolev map for which the set $\NN_u$ of the discontinuity points of $u$ is $\H^1$-negligible and which equals the identity on $\partial\Q$. We say that the function $u$ is \emph{non-crossing on lines (NCL)} if for every injective, Lipschitz curve $\gamma\subseteq\Q\setminus\NN_u$ and for every $\eps>0$ there exists a continuous and injective function $v:\gamma\to\u\Q$, coinciding with the identity on $\partial\Q\cap\gamma$, such that $\|u-v\|_{L^\infty(\gamma)}<\eps$.
\end{defin}

While the NCL property is much simpler and more natural than the NC one, it is the latter which characterises the closure of the diffeomorphisms in the Sobolev norm, thanks to Theorem~\ref{thDPP}. As a consequence, property NCL is in fact useless, unless it coincides with NC. In other words, if NCL implies NC (the other implication is obvious), then it is a very good simplification of the characterisation given by Theorem~\ref{thDPP}. On the contrary, if NCL is strictly weaker than NC, then it is a property which is easy to check but of no use at all. In the next two sections we will notice that in general NCL does not imply NC, and not even INV; but NCL implies (so, it is equivalent to) NC in the physically relevant case of functions $u\in W^{1,p}(\Q;\u\Q)$ for which $\det Du>0$ almost everywhere.

We conclude this section by recalling the pointwise notion of Sobolev functions, in the sense of multifunctions. This is a useful tool, which allows to define in a precise way the ``generalised image'' of every (not almost every!) point.
\begin{defin}
Let $u\in W^{1,1}(\Q;\u\Q)$. For every $x\in\Q$ and every $\eps>0$, we denote by $\D(x,\eps)$ the collection of the open, Lipschitz sets $D\comp B(x,\eps)$ such that the restriction of $u$ to $\partial D$ is continuous. For every such set $D\in\D(x,\eps)$, we also set
\[
U(D) := \Big\{ \u P\in \u\Q:\, {\rm deg}(\u P,u\res\partial D)\neq 0 \Big\} \cup u(\partial D)\,.
\]
Finally, we call \emph{generalised image of $x$ through $u$} the set
\[
u_{\rm multi}(x) := \Big\{ \u P \in \u\Q:\, \exists \,\eps_n\searrow 0,\, D_n\in \D(x,\eps_n),\, \u P_n\in U(D_n):\, \u P=\lim_{n\to\infty} \u P_n\Big\}\,.
\]
\end{defin}
Notice that $u_{\rm multi}(x)$ is a non-empty closed set, which is defined for every $x\in\Q$. The resulting $u_{\rm multi}:\Q \to 2^{\u\Q}$ is then a multi-valued function. Notice that for every continuity point $x\in \Q$ the set $u_{\rm multi}(x)$ reduces to the sole $u(x)$.

\begin{remark}\label{remidb}
Observe that we can always assume without loss of generality that the function $u$ equals the identity not only on the boundary $\QQ$, but also on a small neighborhood of it. Indeed, we can extend $u$ as the identity on a slightly bigger square and then make a homothety.
\end{remark}

\section{The NCL property does not imply INV nor NC\label{sec:ctx}}

This section is devoted to show that the NCL property is not equivalent to the NC one, and in fact it does not even imply INV. In the next section, we will see that NCL becomes equivalent to NC under the additional assumption that $\det Du>0$ a.e.\,.

The whole section consists in presenting and discussing the counterexample, which is an NCL and not {an INV function $u:B_2\to B_2$. We start describing $u$ on $B_2\setminus B_1$. Calling $P=(-1,0)$ and $\mathS=[0,1]\times\{0\}$, the properties which are important are that the restriction of $u$ on $B_2\setminus \overline{B_1}$ is a diffeomorphism onto $B_2\setminus\mathS$, and that
\begin{align*}
u(x)=x\quad \forall\, x\in\partial B_2\,, && u(x)=(1,0)\quad\forall\, x\in\partial B_1\setminus P\,.
\end{align*}
Notice that such a function exists, and it can be constructed in $W^{1,p}(B_2\setminus \overline{B_1})$ for every $1\leq p<2$. One can easily imagine this function as a diffeomorphism which is the identity on $\partial B_2$, and which shrinks the circle $\partial B_1$ onto the segment $\mathS$, compressing the whole $\partial B_1\setminus P$ to the point $(1,0)$. As a consequence, the generalised image of the point $P$ is the whole segment $\mathS$.\par

Let us now pass to describe the function $u$ on $B_1\setminus B_{1/2}$. For every $0\leq \eps\leq 1/2$, the image under $u$ of the circle $\partial B_{1-\eps}$ is the segment $[0,1-2\eps]\times\{0\}$. More precisely, for every $0\leq \theta\leq 2\pi$ we set
\begin{equation}\label{def121}
u\Big((1-\eps)\cos\theta,(1-\eps)\sin\theta)\Big) = \bigg( \frac{1-2\eps}{2\pi\eps}\, |\pi-\theta|\wedge 1,0\bigg)\,.
\end{equation}
Notice that the function $u$ is still continuous outside the point $P$ (actually, an obvious modification would allow to take $u$ smooth outside $P$). In particular, $|Du(x)| \approx \tfrac{1}{|x-P|}$ for $x$ near $P$, and by a simple use of polar coordinates centred at $P$ we have
\[
\int_{B(P,\frac{1}{4})} |Du(x)|^p \approx \int_0^{\frac{1}{4}}r^{1-p} < \infty
\]
 for all $p<2$. Further, since far away from $P$ the derivative $|Du|$ is uniformly bounded, we have $u\in W^{1,p}(B_2\setminus \overline{B_{1/2}})$ for any $1\leq p<2$. Roughly speaking, the images of the circles $\partial B_r$ with radii $r$ between $1$ and $1/2$ are shorter and shorter segments, passing from the whole $\mathS$, corresponding to the radius $r=1$, to the single point $(0,0)$, corresponding to the radius $r=1/2$.\par

Finally, the function $u$ on $B_{1/2}$ is simply defined as
\[
u(x)=\big(1-2|x|,0)\,,
\]
that is, every circle $B_r$ with $0\leq r\leq 1/2$ is sent on a single point, which moves from $(0,0)$ to $(1,0)$ while $r$ decreases from $1/2$ to $0$.

By construction, for any $1\leq p<2$ we have that
\begin{align*}
u\in W^{1,p}(B_2)\,, &&
u(x)=x\quad\forall\, x\in\partial B_2\,, &&
u\in {\rm C}(B_2\setminus \{P\})\,.
\end{align*}
Moreover, we have that
\begin{equation}\label{unotINV}
\hbox{$u$ is not an INV function}\,.
\end{equation}
Indeed, for every $1/2<r<1$ the function $u$ is continuous on $\partial B_r$. Nevertheless, the origin $(0,0)$ is contained inside each of the disks $B_r$, and $u(0,0)=1$ has zero degree with respect to the curve $u\res \partial B_r$.\par

To conclude the example, we have to show that $u$ satisfies the NCL property, and this will take the rest of the section. Since $u$ is continuous on $B_2\setminus \{P\}$, we fix an injective, Lipschitz curve $\gamma:[0,1]\to B_2\setminus\{P\}$ and some $\eps>0$. Our goal is to define a function $\varphi:[0,1] \to B_2$ such that
\begin{align}\label{Linfty01}
\hbox{$\varphi$ is continuous and injective}\,, && \|\varphi- u\circ \gamma\|_{L^\infty([0,1])}<\eps\,.
\end{align}
By the continuity of $u$ on $B_2\setminus \{P\}$ (and by $v \cdot w$ denoting the usual scalar product), without loss of generality we can assume that
\begin{align}\label{ass2}
\gamma\in {\rm C}^\infty([0,1])\,, &&
\gamma(0)\cup \gamma(1) \in \partial B_2\,, &&
\gamma'(t)\cdot \gamma(t)\neq 0\quad \forall \, t\in \gamma^{-1}(\partial B_1)\,.
\end{align}
The last assumption implies that, whenever it meets the circle $\partial B_1$, the curve $\gamma$ is actually entering in the disk $B_1$, or exiting from it. As a consequence, there exists some $N\in\N$ and points $0< a_1<b_1<a_2<b_2<\, \cdots \, < a_N<b_N<1$ such that
\[
\gamma^{-1}(B_1)= \cup_{i=1}^N (a_i,b_i)\,,
\]
while $u^{-1}(\partial B_1)$ only consists precisely of the union of the points $a_i$ and $b_i$. By compactness and~(\ref{ass2}), it is possible to select $\delta\ll \eps$ so that $\gamma([0,1])\cap \partial B_{1+\delta}$ consists of $2N$ points $0< a_1^-<b_1^+<a_2^-<b_2^+<\, \cdots \, < a_N^-<b_N^+<1$, with $(a_i,b_i)\comp (a_i^-,b_i^+)$ for every $1\leq i\leq N$, see a possible example in Figure~\ref{Fig:AiBi}, left.
\begin{figure}[thbp]
\begin{align*}
\scalebox{0.7}{
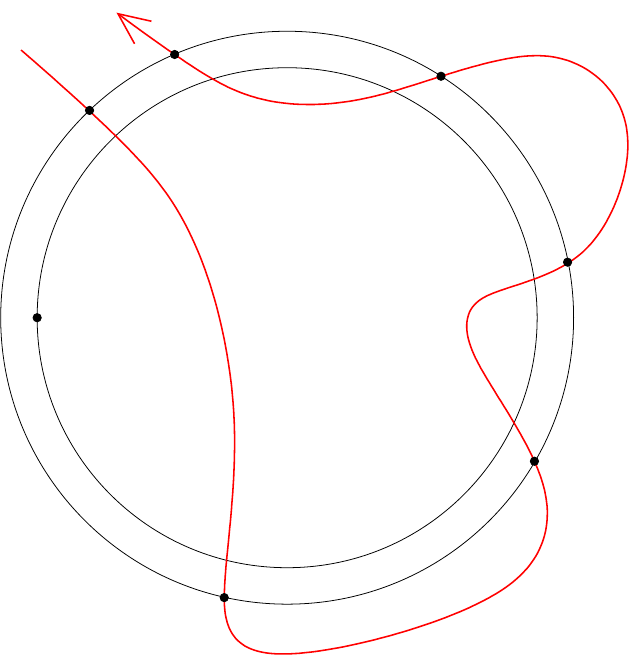
}
\quad&&\quad
\scalebox{0.7}{
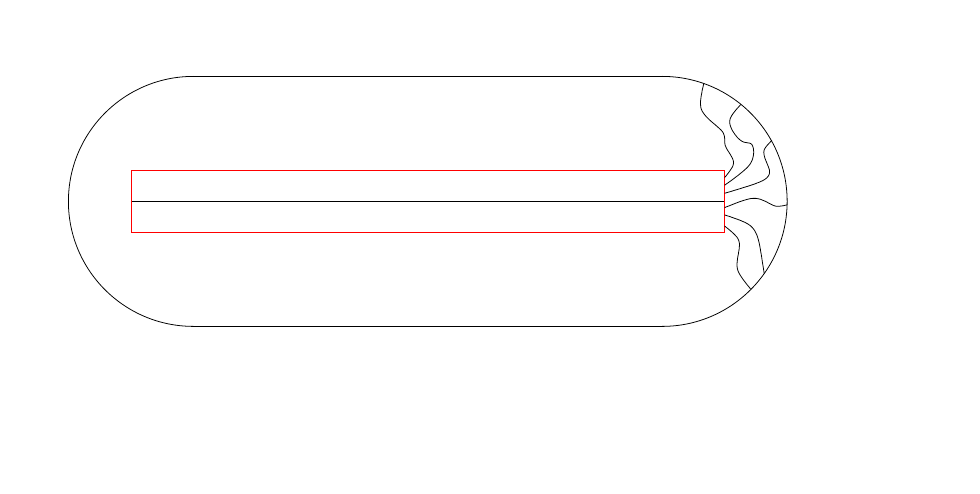
}
\end{align*}
\caption{Example of a possible curve $\gamma$ near $B_1$ and image under $\varphi$ of the intervals $(a_i^-,a_i)$ and $(b_i,b_i^+)$.}\label{Fig:AiBi}
\end{figure}
For any such $1\leq i\leq N$, we define
\begin{align*}
P_i = \gamma(a_i^-)\,, && Q_i=\gamma(b_i^+)\,, &&
\u P_i=u(P_i)\,, && \u Q_i=u(Q_i)\,.
\end{align*}
Since $u$ is injective outside of $\overline{B_1}$, the $P_i$'s and the $Q_i$'s are $2N$ distinct points in $\partial B_i\setminus \{P\}$, and the chords $P_iQ_i$ are all disjoint. Notice that by construction, up to take $\delta$ small enough, the Hausdorff distance between $\C_\delta:=u(\overline{B_{1+\delta}})$ and the segment $\mathS$ is $d_H(\C_\delta,\mathS)\ll\eps$, as well as the distance between each of the points $\u P_i,\, \u Q_i$ and the point $(1,0)$ --keep in mind that $u$ shrinks the whole $\partial B_1\setminus \{P\}$ on the point $(1,0)$.\par

Let us now start constructing the function $\varphi$. First of all, we let $\varphi=u\circ \gamma$ on the set $[0,1]\setminus \cup_{i=1}^N (a_i^-,b_i^+)$. Of course, where it has been already defined, $\varphi$ satisfies~(\ref{Linfty01}).\par

We now pass to define $\varphi$ on the union $U$ of the intervals $[a_i^-,a_i]$ and $[b_i,b_i^+]$. To do so, we let $\delta'\ll \delta$ be a constant such that the open rectangle $\mathS^+=(0,1)\times (-\delta',\delta')$ is compactly contained in $\C_\delta$. Moreover, we select $2N$ distinct points $\widetilde{\u P}_i,\, \widetilde{\u Q}_i$ on the segment $\{1\} \times (-\delta',\delta')$, which is the right side of $\mathS^+$, in such a way that the order of the points $\gamma(a_i)$ and $\gamma(b_i)$ on $\partial B_1\setminus \{P\}$) is the same as the order of the points $\widetilde{\u P}_i,\, \widetilde{\u Q}_i$ on the segment $\{1\} \times (-\delta',\delta')$. It is then possible to define $\varphi$ as a continuous, injective function on $U$ in such a way that for every $1\leq i\leq N$
\begin{align*}
\varphi(a_i^-)=\u P_i\,, &&
\varphi(a_i)=\widetilde{\u P}_i\,, &&
\varphi(b_i)=\widetilde{\u Q}_i\,, &&
\varphi(b_i^+)=\u Q_i\,,
\end{align*}
and that
\begin{align*}
\varphi(U)\subseteq \C_\delta\setminus \mathS^+\,,&&
\varphi(U\setminus \cup_{i=1}^N \{a_i,\,b_i\})\subseteq (\C_\delta)^\circ\setminus \overline{\mathS^+}\,,&&
d_H\big(\varphi(U),(1,0)\big)\ll \eps\,.
\end{align*}
It is still true that, where it has been already defined, $\varphi$ satisfies~(\ref{Linfty01}). The situation is depicted in Figure~\ref{Fig:AiBi}, right.\par

To conclude, we have to define $\varphi$ in each interval $[a_i,b_i]$. To do so, we recall that $u$ is continuous on $B_2\setminus \{P\}$ and that $u(B_1)=\mathS$. As a consequence, for every $1\leq i \leq N$ we have
\[
u\big(\gamma([a_i,b_i])\big) = [\ell_i, 1]\times \{0\}
\]
for some $0\leq \ell_i<1$. Since, as noticed before, all the chords $P_iQ_i$ are disjoint, then for every $1\leq i,\, j\leq N$ with $i\neq j$ the two segments $\widetilde{\u P}_i\widetilde{\u Q}_i$ and $\widetilde{\u P}_j\widetilde{\u Q}_j$ are either disjoint or contained into one other. We have the following property.
\begin{lemma}\label{leleft}
Let $1\leq i,\,j\leq N$ be two distinct indices such that $\widetilde{\u P}_i\widetilde{\u Q}_i\subseteq \widetilde{\u P}_j\widetilde{\u Q}_j$. Then, $\ell_j\leq \ell_i$.
\end{lemma}
\begin{proof}
If $\ell_j=0$ there is nothing to prove, so we assume that $\ell_j>0$. Since, by~(\ref{def121}), $u(x)=(0,0)$ for every $x\in \partial B_{1/2}$, as well as for every $x\in [-1,-1/2]\times\{0\}$, the assumption $\ell_j>0$ implies that $\gamma((a_j,b_j))$ is a continuous curve inside $B_1$ with endpoints $P_j$ and $Q_j$ which divides $B_1$ in two parts, one of which contains both $P$ and the ball $B_{1/2}$, while the other contains the curve $\gamma((a_i,b_i))$, since $\gamma$ is injective.\par
As a consequence, for any point $R\in \gamma((a_i,b_i))$, the segment joining $R$ with the origin intersects $\gamma((a_j,b_j))$. Let then $R$ be any point in $\gamma((a_i,b_i))$, and let $\theta\in [0,2\pi]$ and $0\leq \eps< 1/2$ be such that $R=\big((1-\eps)\cos\theta,(1-\eps)\sin\theta\big)$. There exists some $\eps<\eps'<1/2$ such that the point $R'=\big((1-\eps')\cos\theta,(1-\eps')\sin\theta\big)$ belongs to $\gamma((a_j,b_j))$. Keeping in mind~(\ref{def121}), we derive that $u(R')$ is on the left of $u(R)$. Since $R\in \gamma((a_i,b_i))$ is generic, the thesis follows.
\end{proof}

Let us now proceed with our definition. We let $\delta''\ll\delta'$ be a constant much smaller than the distance between any two points in the set $\big\{\widetilde{\u P}_i,\,\widetilde{\u Q}_i,\, 1\leq i\leq N\big\}$. Then, we define $N$ distinct constants $\tilde \ell_i$ so that for every $i$ one has $|\tilde\ell_i-\ell_i|< \delta''$ and in addition, whenever $\widetilde{\u P}_i\widetilde{\u Q}_i\subseteq \widetilde{\u P}_j\widetilde{\u Q}_j$, one has $\tilde\ell_j < \tilde\ell_i$. This is possible thanks to Lemma~\ref{leleft}. Now, for every $1\leq i \leq N$ we select $c_i\in (a_i,b_i)$ so that $u_1(\gamma(c_i))=\ell_i$, where we write $u=(u_1,u_2)$. Keep in mind that $u_2\circ\gamma\equiv 0$ on the intervals $[a_i,b_i]$. And finally, we select $\delta'''\ll \delta''$ so that, for every $1\leq i\leq N$, one has $(c_i-2\delta''',c_i+2\delta''')\comp (a_i,b_i)$ and $u_1\circ \gamma (c_i-2\delta''',c_i+2\delta''') \comp [\tilde\ell_i-\delta'',\tilde\ell_i+\delta'']$.\par

We are now in position to define $\varphi$ on the intervals $[a_i,b_i]$. Let us fix $1\leq i\leq N$, and keep in mind that
\[
u_1(\gamma(a_i))=u_1(\gamma(b_i))=1 = \varphi_1(a_i)=\varphi_1(b_i)\,.
\]
Let us assume for a moment, just to fix the ideas, that $\varphi_2(a_i)>\varphi_2(b_i)$, that is, the point $\widetilde{\u P}_i$ is above $\widetilde{\u Q}_i$. In the interval $[a_i,c_i-2\delta''']$, we let then $\varphi$ be the function such that $\varphi_1 = u_1\circ\gamma$, and $\varphi_2$ is affine with $\varphi_2(c_i-2\delta''')=\varphi_2(a_i)-\delta''$. Similarly, in the interval $[c_i+2\delta,b_i]$ we let $\varphi$ be the function such that $\varphi_1=u_1\circ\gamma$, and $\varphi_2$ is affine with $\varphi_2(c_i+2\delta)=\varphi_2(b_i)+\delta''$. In words, in the interval $[a_i,c_i-2\delta''']$ (resp., $[c_i+2\delta''',b_i]$) the function $\varphi$ is behaving exactly as $u$ horizontally, while vertically it is slowly going down (resp., up). Then, in the interval $[c_i-2\delta''',c_i-\delta''']$, as well as in the interval $[c_i+\delta''',c_i+2\delta''']$, the function $\varphi$ is affine with
\begin{align*}
\varphi(c_i-\delta''')=\big(\tilde\ell_i, \varphi_2(a_i)-2\delta''\big)\,, && \varphi(c_i+\delta''')=\big(\tilde\ell_i, \varphi_2(b_i)+2\delta''\big)\,.
\end{align*}
And finally, in the interval $[c_i-\delta''',c_i+\delta''']$, the function $\varphi$ is affine, hence its image is the vertical segment $\{\tilde\ell_i\}\times [\varphi_2(b_i)+2\delta'',\varphi_2(a_i)-2\delta'']$. The construction of $\varphi$ in the interval $[a_i,b_i]$ is depicted in Figure~\ref{a_ib_i}. If $\widetilde{\u P}_i$ is below $\widetilde{\u Q}_i$, so $\varphi_2(a_i)<\varphi_2(b_i)$, the definition of $\varphi$ is exactly the same, except that the second coordinate of $\varphi$ is slowly increasing from $a_i$ to $b_i$, instead of slowly decreasing.\par

\begin{figure}[thbp]
\scalebox{1}{
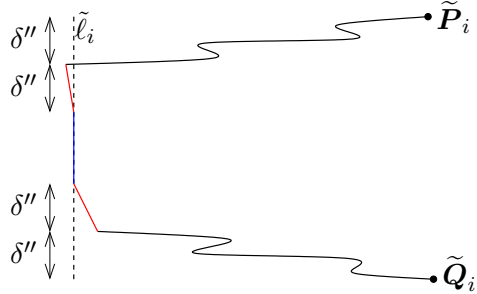
}
\caption{Construction of the function $v$ in the interval $[a_i,b_i]$.}\label{a_ib_i}
\end{figure}

It is easy to observe that, in the whole $[0,1]$, the constructed function $\varphi$ satisfies~(\ref{Linfty01}). Hence, we have proved that $u$ is an NCL function, so by~(\ref{unotINV}) our example is concluded.

\begin{remark}
The example that we presented in this section is an NCL function which is not INV, thus also not NC. It is also possible to provide a function which is both NCL and INV, but still not NC. For instance, this is what happens to the function presented in~\cite[Section~5.2]{DPP}. It is already proved there that such function is INV and not NC. The proof that it is also NCL is very similar to what we have done in this section, in particular for that function the analogous of Lemma~\ref{leleft} holds.
\end{remark}

\section{The NCL property together with $\det Du > 0$ a.e. implies NC\label{sec:bright}}

This section is devoted to show that, under the assumption that $\det Du>0$ a.e., the properties NC and NCL are equivalent. Since the NC property clearly implies the NCL one, we basically have to show that if the determinant of $Du$ is strictly positive almost everywhere then NCL is a sufficient condition to get NC. The result is then the following one.

\begin{theorem}[NCL and $\det Du>0$ a.e. $\Longrightarrow NC$]\label{NCL>0NC}
Let $u\in W^{1,1}(\QQ;\u\QQ)$ be an NCL function coinciding with the identity on $\partial\QQ$ and such that $\det Du>0$ a.e.\,. Then, $u$ is also NC.
\end{theorem}

\begin{remark}\label{thatsobv}
Notice that the assumption that the function $u$ coincides with the identity on $\partial\QQ$ is just made for simplicity. Of course, given any bi-Lipschitz bijection $\Phi:\R^2\to\R^2$, it is completely equivalent to consider functions in $W^{1,1}(\QQ;\Phi(\QQ))$ which coincide with $\Phi$ on the boundary.
\end{remark}

To show this result, we need to introduce the concept of good starting grid and good arrival grid. These concepts were already used in~\cite{DPP}, our definition is a bit simpler since we do not need here all the properties which were needed there.

\begin{defin}[Good starting and arrival grids]
Let $u\in W^{1,1}(\QQ;\u\QQ)$ be a Sobolev function, let $\NN_u$ be the set of its discontinuity points, and let $\gamma:[0,1]\to\QQ\setminus \NN_u$ be a Lipschitz curve. We say that $\gamma$ is an \emph{admissible curve} if $\gamma(t)$ is a Lebesgue point for $Du$ for $\H^1$-a.e. $t$, and the curve $u\circ \gamma$ belongs to $W^{1,1}([0,1];\u\QQ)$.\par
Given $K\in\N$, we call \emph{$K$-grid} the set $\G=\G(K)$ defined as
\[
\G(K) = \bigcup_{i=0}^K [0,1] \times \{i/K\} \ \cup\ \bigcup_{j=0}^K \{j/K\} \times [0,1] \subseteq\QQ\,,
\]
and we say that the grid $\G=\G(K)$ is a \emph{good starting grid} if each segment contained in $\G$ is an admissible curve.\par
Given a good starting grid $\G=\G(K)$, a small constant $\eta\ll 1/K$, and some coordinates $0=x_0<x_1<x_2\, \cdots\, ,\, x_N=1$ and $0=y_0<y_1<y_2\, \cdots < y_M=1$ with $x_{n+1}-x_n<\eta$ and $y_{m+1}-y_m<\eta$ for every $0\leq n< N$ and $0\leq m<M$, we say that
\begin{equation}\label{defgag}
\widetilde\G = \bigcup_{n=1}^N \{x_n\}\times [0,1]\ \cup\ \bigcup_{m=0}^M [0,1] \times \{y_m\}\subseteq\u\QQ
\end{equation}
is a \emph{good arrival grid associated with $\G$ and with side-length $\eta$} if $u^{-1}(\widetilde\G)\cap \G\cap \QQ^\circ$ is done by finitely many points, each of which is a Lebesgue point of $Du$ and not a vertex of the grid $\G$, and moreover for any such point $P\in u^{-1}(\widetilde\G)\cap \G\cap \QQ^\circ$ the point $\u P=u(P)$ is not a vertex of the grid $\widetilde\G$ and, calling $\tau\in\S^1$ the direction of the side of $\G$ containing $P$, the tangential derivative $D_\tau(P)$ is not parallel to the side of $\widetilde\G$ containing $\u P$.
\end{defin}

An important fact is that good arrival grids always exist. More precisely, we have the following property, whose proof is a simple variant of the proof of~\cite[Lemma~3.6]{DPP}.

\begin{lemma}\label{approx36}
Let $u\in W^{1,1}(\QQ;\u\QQ)$, let $\G=\G(K)$ be a good starting grid, let $\Sigma\subseteq\G$ be a $\H^1$-negligible set, and let $\eta\ll 1/K$ be fixed. Then, there exists a good arrival grid $\widetilde\G$ associated with $\G$, with side-length $\eta$, and such that $u^{-1}(\widetilde\G)\cap \Sigma=\emptyset$.
\end{lemma}
\begin{proof}
Let us call $A$ the set of vertices of $\G$, i.e. the points $(p/K,q/K)$ with $1\leq p,\,q\leq K-1$, and let $\u A^1$ and $\u A^2$ be the horizontal and the vertical projection of $u(A)$. Notice that $\u A^1$ (resp., $\u A^2$) is a finite set of abscissae (resp. ordinates), since $A$ is a finite set and $u$ is continuous at any of its points.\par

Let now $\gamma$ be a horizontal or vertical segment contained in $\G$, with both endpoints in $\partial\QQ$. Let us define $B_\gamma$ the set of points of $\gamma$ which are not Lebesgue points for $Du$, or which belong to $\Sigma$. By definition, $B_\gamma$ is a $\H^1$-negligible subset of $\gamma$. Since the restriction of $u$ to $\gamma$ belongs to $W^{1,1}$, also $u(B_\gamma)\subseteq \u\QQ$ is $\H^1$-negligible, hence so are also its projections $\u B^1_\gamma$ and $\u B^2_\gamma$.\par

Let then $C^1_\gamma$ (resp., $C^2_\gamma$) the set of points $P$ in $\gamma\setminus B_\gamma$ such that the first (resp., second) component of $D_\tau u(P)$ is zero, where $\tau$ is the direction of $\gamma$, so either $(1,0)$ or $(0,1)$. The sets $C^1_\gamma$ and $C^2_\gamma$ need not to be negligible. Nevertheless, the first projection $\u C^1_\gamma$ of $u(C^1_\gamma)$ is $\H^1$-negligible, as well as the second projection $\u C^2_\gamma$ of $u(C^2_\gamma)$. This is a very standard fact, the easy proof can be found for instance in~\cite[Lemma~3.6]{DPP}.\par

We call now $\u B^1$ (resp., $\u B^2$, $\u C^1$, $\u C^2$) the union of the sets $\u B^1_\gamma$ (resp., $\u B^2_\gamma$, $\u C^1_\gamma$, $\u C^2_\gamma$) for all the horizontal and vertical segments $\gamma$ contained in $\G$. Since these segments are finitely many, we have
\[
\H^1(\u B^1)=\H^1(\u B^2)=\H^1(\u C^1)=\H^1(\u C^2)=0\,.
\]

Let now take any $x\in (0,1)\setminus (\u A^1\cup \u B^1\cup \u C^1)$, and take some point $P\in u^{-1}(\{x\}\times [0,1])\cap \G$. Since $x\notin \u A^1$ the point $P$ is not a vertex of $\G$, thus there is only one segment $\gamma$ contained in $\G$ and with endpoints in $\partial\QQ$ which contains $P$. Since $x\notin (\u B^1_\gamma\cup \u C^1_\gamma)$, the point $P$ does not belong to $\Sigma$ and it is a Lebesgue point for $Du$, and $D_\tau u (P)$ has a non-zero first component, being $\tau$ again the direction of the segment $\gamma$. This means that $P$ is the unique point of $u^{-1}(\{x\}\times [0,1])$ in a suitably small neighborhood of $P$ in $\gamma$, hence the set $u^{-1}(\{x\}\times [0,1])\cap \G$ is finite.\par

We can then select numbers $0=x_0<x_1 < x_2 \,\cdots\, < x_{N-1}<x_N=1$ such that no $x_i$ belongs to $\u A^1\cup \u B^1\cup \u C^1$, except $x_0$ and $x_N$, and such that $x_{i+1}-x_i < \eta$ for every $0\leq i<N$. By construction, the set
\[
D= \G \cap \bigcup_{i=1}^{N-1} u^{-1}\big(\{x_i\} \times (0,1) \big)
\]
is finite, and since $u$ is continuous on every point of $\G$ we deduce that also $\u D=u(D)$ is finite.\par

The very same argument as before implies that for any $y\in (0,1)\setminus (\u A^2\cup \u B^2\cup \u C^2)$ the set $u^{-1}([0,1]\times \{y\})\cap \G$ is finite. Hence, we can select numbers $0=y_0<y_1< y_2\, \cdots\, < y_{M-1}<y_M=1$ such that no $y_j$ belongs to $\u A^1\cup \u B^1\cup \u C^1$, except $y_0$ and $y_M$, that $y_{j+1}-y_j<\eta$ for every $0\leq j<M$, and also that no $y_j$ belongs to the second projection of the finite set $\u D$.\par

Using the coordinates $\{x_i\}$ and $\{y_i\}$ to define the grid $\widetilde\G$ as in~(\ref{defgag}), the fact that $\widetilde\G$ is a good arrival grid associated to $\G$ and with side-length $\eta$ is true by constrution. In particular, the fact that the points of $\widetilde\G\cap u(\G)\cap\u\QQ^\circ$ are not vertices of $\widetilde \G$ is true because the $y_j$ have been chosen not to belong to the second projection of $\u D$.
\end{proof}

A simple but fundamental technical fact needed to prove Theorem~\ref{NCL>0NC} is the following.
\begin{lemma}\label{gmmg}
Let $u\in W^{1,1}(\QQ;\u\QQ)$ be an NCL function, coinciding with the identity on $\partial\QQ$, and let $x\neq y \in\QQ\setminus\NN_u$ be two Lebesgue points for $Du$ with $\det Du(x)>0$ and $\det Du(y)>0$. Then, $u(x)\neq u(y)$.
\end{lemma}
\begin{proof}
Let us assume, by contradiction, that $u(x)=u(y)$. Moreover, without loss of generality and for simplicity of notation, let us assume that $u(x)=u(y)=(0,0)\in\u\Q$. For brevity of notation we write $\M_x=Du(x)$ and $\M_y=Du(y)$, and we set
\begin{align*}
\omega_x(z)=u(x)+ \M_x(z-x)\,, && \omega_y(z)=u(y)+ \M_y(z-y)\,.
\end{align*}
Since $\det \M_x\neq 0$ and $\det\M_y\neq 0$, and since $\NN_u$ is $\H^1$-negligible and it does not contain $x$ nor $y$, then up to a small rotation we can assume that
\begin{equation}\label{goodhorver}
\omega_x^{-1}\big(\{0\}\times\R\big) \cup \omega_y^{-1}\big(\R\times\{0\}\big) \subseteq \R^2 \setminus\NN_u\,.
\end{equation}
Observe that $\omega_x^{-1}(\{0\}\times\R)$ and $\omega_y^{-1}(\R\times\{0\})$ are two lines passing through $x$ and $y$ respectively.\par

Since $x$ is a Lebesgue point for $Du$, it is simple to observe that
\begin{equation}\label{Linftyclose}
\forall \, \eps\ \exists\, \bar r=\bar r(\eps)< {\rm dist}(x,\partial\QQ)\, :\ \forall\, r<\bar r,\ \| u-\omega_x \|_{L^\infty(B(x,r))} < \eps r\,,
\end{equation}
see for instance~\cite[Lemma~4.3]{HP}, or~\cite[Lemma~4.2]{PR} for mappings of bounded variation.\par

Since $\det \M_x>0$, we have $\eps_1:=\min\{\M_x(v)|,\, |v|=1\}>0$. Let now $\ell\ll 1$ be fixed, and let us consider the segment $\u V=\{0\}\times [-\ell,\ell]\in\u\QQ$ and the corresponding $V=\omega_x^{-1}(\u V)$, which is a segment centered at $x$ on which $u$ is continuous by~(\ref{goodhorver}). The estimate~(\ref{Linftyclose}) yields that, if
\[
\ell < \eps_1 \bar r(\eps_1/2)\,,
\]
then
\begin{equation}\label{uom1}
\|u-\omega_x\|_{L^\infty(V)} \leq \frac\ell 2\,.
\end{equation}
The very same argument applied to $y$ in place of $x$ implies that, up to decrease $\ell$ if necessary, calling $\u H=[-\ell,\ell]\times\{0\}\in\u\Q$ and $H=\omega_y^{-1}(\u H)$ we have that $H$ is a segment centered at $y$ on which $u$ is continuous, and
\begin{equation}\label{uom2}
\|u-\omega_y\|_{L^\infty(H)} \leq \frac\ell 2\,.
\end{equation}
Again up to decrease $\ell$, we can assume that $V$ and $H$ are two disjoint segments, since they are arbitrarily short and centered at the two distinct points $x$ and $y$.\par

Let then $\gamma:[0,1]\to\QQ\setminus\NN_u$ be any Lipschitz curve whose image contains both $V$ and $H$. Since $u$ is NCL, there exists a continuous and injective curve $\varphi:[0,1]\to \u\QQ$ such that
\[
\|u\circ \gamma - \varphi\|_{L^\infty[0,1]} < \frac \ell 2\,,
\]
and this is clearly a contradiction with~(\ref{uom1}) and~(\ref{uom2}).
\end{proof}

We are now in position to prove Theorem~\ref{NCL>0NC}.

\proofof{Theorem~\ref{NCL>0NC}}
Let $u\in W^{1,1}(\QQ;\u\QQ)$ be an NCL function coinciding with the identity on $\partial\QQ$ and with $\det Du>0$ a.e.\,. To prove that $u$ is an NC function, we take a collection $\Gamma$ of curves as in Definition~\ref{defNC} and $\eps>0$, and we have to find some continuous and injective function $v:\Gamma\to\u\QQ$, coinciding with the identity on $\Gamma\cap\partial\QQ$, and such that $\|v-u\|_{L^\infty(\Gamma)}<\eps$. We divide the proof in three steps.
\step{I}{Reduction to the case $\Gamma=\G(K)$.}
In this first step, we observe that it is enough to consider the case when $\Gamma=\G(K)$. In the next steps we will prove the claim in this case.\par

First of all, since $\Gamma$ is a finite union of Lipschitz curves not intersecting $\NN_u$, there is a bi-Lipschitz bijection $\Phi:\QQ\to\QQ$ such that $\Phi(\Gamma)\subseteq \G(K)$ for a sufficiently large $K$. If we call $\overline u=u\circ \Phi^{-1}$, we clearly have that $\overline u\in W^{1,1}(\QQ;\u\QQ)$ is an NCL function coinciding with $\Phi^{-1}$ on $\partial\QQ$ and satisfying $\det D\overline u>0$ a.e.\,. Notice that $\Phi(\Gamma)$ does not intersect $\NN_{\overline u}$. As a consequence, since $\NN_{\overline u}$ is $\H^1$-negligible, it is possible to modify $\Phi$ so to get the additional property that the whole $\G(K)$ does not intersect $\NN_{\overline u}$. Assuming then that the claim is true for the special case when $\Gamma=\G(K)$, and keeping in mind Remark~\ref{thatsobv} we can find a function $\overline v:\G(K)\to\u\QQ$ which is continuous, injective, coincides with $\Phi^{-1}$ on $\partial\QQ$, and
\[
\| \overline v - \overline u \|_{L^\infty(\smG(K))} < \eps\,.
\]
We can then define $v=\overline v \circ \Phi$. This is a continuous, injective function defined on $\Phi^{-1}(\G(K))$, which coincides with the identity on $\partial\QQ$, and which satisfies
\[
\| v - u \|_{L^\infty(\Phi^{-1}(\smG(K)))}=\| \overline v - \overline u \|_{L^\infty(\smG(K))} < \eps\,.
\]
Since $\Phi^{-1}(\G(K))\supseteq \Gamma$, we conclude the thesis.

\step{II}{Reduction to the case $\H^1(\G(K)\cap\Sigma^+)=0$.}
We have to show the result in the special case when $\Gamma=\G(K)$. In this step, we further reduce ourselves to the case when $\G(K)$ is a good starting grid and $\G(K)\cap\Sigma^+$ is $\H^1$-negligible, where $\Sigma^+\subseteq\QQ$ is the set of points $x\in\QQ$ such that either $x$ is not a Lebesgue point for $Du$, or $\det Du(x)=0$. Keep in mind that $|\Sigma^+|=0$ by assumption, hence for a.e. $s\in [0,1]$ and for a.e. $t\in[0,1]$ we have
\begin{equation}\label{trueae}
\begin{array}{cc}
\H^1\big(\{s\}\times [0,1]\cap\Sigma^+\big)=0\,,\qquad &\qquad \H^1\big([0,1]\times \{t\}\cap\Sigma^+\big)=0\,, \\
u\in W^{1,1}\big(\{s\}\times [0,1]\big)\,,\qquad & \qquad u\in W^{1,1}\big([0,1] \times \{t\}\big)\,, \\
\big(\{s\}\times [0,1]\big)\cap \NN_u=\emptyset\,,\qquad & \qquad \big([0,1]\times \{t\}\big)\cap \NN_u=\emptyset\,.
\end{array}
\end{equation}
As a consequence, we can find $0=s_0<s_1<s_2 \cdots < s_K=1$ and $0=t_0<t_1<t_2 \cdots < t_K=1$, with every $s_i$ (resp., $t_j$) arbitrarily close to $i/K$ (resp., $j/K$), such that~(\ref{trueae}) is true for every $s=s_i$ and every $t=t_j$. Since $\G(K)\cap \NN_u=\emptyset$, we can select these $\{s_i\}$'s and $\{t_j\}$'s in such a way that the restriction of $u$ to any horizontal or vertical segment of $\G(K)$ is arbitrarily close to its restriction in the corresponding horizontal segment $[0,1]\times \{t_j\}$ or vertical segment $\{s_i\}\times [0,1]$. The thesis follows then immediately under the assumption that the claim is true when $\Gamma=\G(K)$, $\G(K)$ is a good starting grid, and $\H^1(\G(K)\cap\Sigma^+)=0$.

\step{III}{The case when $\Gamma=\G(K)$ and $\H^1(\G(K)\cap\Sigma^+)=0$.}
Thanks to Step~I and Step~II, we can now conclude the proof by only considering the special case in which $\Gamma=\G(K)$ is a good starting grid and $\H^1(\Sigma^+\cap\G(K))=0$.\par

Lemma~\ref{approx36} with $\Sigma=\Sigma^+\cap\G(K)$ provides us with a good arrival grid $\widetilde\G$ associated with $\G$ and with side-length $\eta<\eps/\sqrt 2$. Let us call $0=x_0<x_1<x_2\, \cdots\, ,\, x_N=1$ and $0=y_0<y_1<y_2\, \cdots < y_M=1$ the coordinates associated with $\widetilde\G$ as in~(\ref{defgag}). Keep in mind that the set $\G\cap u^{-1}(\widetilde\G)\cap\QQ^\circ$ is finite and none of its points belongs to $\Sigma$. Let us then consider the finite set made by $\G\cap u^{-1}(\widetilde\G)\cap\QQ^\circ$ together with all the points of the form $(p/K,q/K)$ with either $p\in \{0,\, K\}$ and $0\leq q\leq K$, or $q\in \{0,\,K\}$ and $0<p<K$, and let us call $P_1,\, P_2,\, \dots \, , \, P_H$ its elements. Moreover, let us call $\u P_i=u(P_i)$ their images. Notice that the points $\u P_i$'s are well defined because every $P_i$ belongs to $\G(K)$, so it is a continuity point of $u$, and in particular $\u P_i=P_i$ for all the points $P_i$ which are on $\partial\QQ$. Observe also that the points $\u P_i$'s are \emph{all distinct} by Lemma~\ref{gmmg} and since every $P_i$ does not belong to $\Sigma$. In fact, the points of $\G\cap u^{-1}(\widetilde\G)\cap\QQ^\circ$ do not belong to $\Sigma$ by Lemma~\ref{approx36}, and the other points $P_i$ are on the boundary of $\QQ$, so they do not belong to $\Sigma$ because $u$ equals the identity on a small neighborhood of the boundary (we can always assume this without loss of generality, as noted in Remark~\ref{remidb}).\par

For every $0\leq n<N$ and $0\leq m<M$, let us define the rectangle
\[
\u\QQ_{n,m}:=[x_n,x_{n+1}]\times [y_m,y_{m+1}]\,.
\]
Let also $\Phi_{n,m}:\u\QQ_{n,m}\to \overline{B_1}$ be a bi-Lipschitz bijection between $\u\QQ_{n,m}$ and the closure of the unit ball. For every $\u P,\,\u Q\in \u\QQ_{n,m}$, we let $[\u P,\u Q]=\Phi_{n,m}^{-1}(\mathS)$, where $\mathS\subseteq \overline{B_1}$ is the closed segment joining $\Phi_{n,m}(\u P)$ and $\Phi_{n,m}(\u Q)$. The choice of the particular map $\Phi_{n,m}$ makes no difference, what is important is that $[\u P,\u Q]$ is a sort of ``generalised segment'' between $\u P$ and $\u Q$. More precisely, it is a path joining $\u P$ and $\u Q$ which is completely in the interior of the rectangle, except possibly for the two endpoints. In particular, given four ordered points $\u P,\,\u Q,\,\u R,\,\u S\in\partial\u\QQ_{n,m}$, the generalised segments $[\u P,\u R]$ and $[\u Q,\u S]$ have exactly one point in common, while $[\u P,\u Q]$ and $[\u R,\u S]$ are disjoint, as well as $[\u P,\u S]$ and $[\u Q,\u R]$.\par

We are now in position to build the map $v:\G\to\u\QQ$. First of all, for any point $1\leq i\leq H$ we set $v(P_i)=\u P_i$. Let now $i,\,j$ be two distinct indices in $\{1,\,2,\, \dots\,,\,H\}$. We say that the segment $P_iP_j$ is a \emph{simple segment} if it is contained in the grid $\G$, and it does not contain any other point $P_k$ with $k\notin \{i,\,j\}$. Notice that the grid $\G$ is the essentially disjoint union of finitely many simple segments, and moreover every $P_i$ is endpoint of exactly two simple segments. By construction, the image under $u$ of any simple segment is entirely contained in some rectangle $\u\QQ_{n,m}$ (the image is actually contained in the \emph{interior} of $\u\QQ_{n,m}$ except for the two endpoints, unless the simple segment is contained in $\partial\QQ$). We let then the image of the simple segment $P_iP_j$ under $v$ be the generalised segment $[\u P_i,\u P_j]$. Hence, to conclude the definition of $v$ we only have to specify the parameterisation of $v$ on each simple segment.\par

Before doing so we observe that, regardless of the parameterisation that we will set, in any case we will have that for every point $P\in \G$ the two images $u(P)$ and $v(P)$ are in a same rectangle. Therefore, since the grid $\widetilde \G$ has side-length $\eta<\eps/\sqrt 2$, we will have
\[
\|u-v\|_{L^\infty(\smG)} \leq \sqrt 2\, \eta < \eps\,.
\]
In other words, the parameterisation of $v$ has only to be chosen in such a way that the resulting $v$ is continuous, injective, and coincides with the identity on $\partial \QQ$.\par

To specify the parameterisation of $v$, we need first to select the image of the \emph{vertices} of $\G$, i.e. the points of the form $V_{p,q}=(p/K,q/K)$ with $1\leq p,\,q \leq K-1$. Let $V=V_{p,q}$ be any such vertex. By construction, $V$ is the intersection between two simple segments, a horizontal one and a vertical one. Just for simplicity of notation, let us think that the horizontal simple segment is $P_1P_2$, and the vertical one is $P_3P_4$, with $P_1$ on the left of $P_2$ and $P_3$ above $P_4$. The points $\u P_1,\,\u P_2,\, \u P_3$ and $\u P_4$ are then four distinct points on the boundary of a same rectangle $\u\QQ_{n,m}$. A fundamental fact is the following.

\claim{1}{The points $\u P_2$ and $\u P_4$ are on the same one of the two connected components in which $\partial\u\QQ_{n,m}$ is divided by $\u P_1$ and $\u P_3$.}
To show the claim, we start defining two curves $\gamma_a$ and $\gamma_b$, one connecting $P_1$ and $P_3$, and the other one connecting $P_4$ and $P_2$. More precisely, $\gamma_a$ is done by three segments; a horizontal one, connecting $P_1$ and a point $V_{\rm left}$ which is contained in the interior of $P_1 V$; then a diagonal one, connecting $V_{\rm left}$ and a point $V_{\rm up}$ in the interior of $V P_3$; and then a vertical one, connecting $V_{\rm up}$ and $P_3$. Similarly, $\gamma_b$ is done by a vertical segment connecting $P_4$ and some $V_{\rm down}$ in the interior of $P_4 V$, together with a diagonal one connecting $V_{\rm down}$ and some point $V_{\rm right}$ in $V P_2$, and a horizontal one connecting $V_{\rm right}$ and $P_2$. The curves $\gamma_a$ and $\gamma_b$ are clearly disjoint, and they are admissible for almost every choice of the four points $V_{\rm left},\, V_{\rm right},\, V_{\rm up}$ and $V_{\rm down}$. Since $V$ is a continuity point for $u$, and $u(V)$ is in the interior of $\u\QQ_{n,m}$, we can select the four points arbitrarily close to $V$, in such a way that not only $\gamma_a$ and $\gamma_b$ are disjoint and admissible, but in addition their image under $u$ is entirely contained in the interior of $\u\QQ_{n,m}$ except for the four endpoints. We let then $\gamma:[0,1]\to\QQ$ be a Lipschitz, injective, admissible curve, containing in its interior both $\gamma_a$ and $\gamma_b$, and being affine for a while around the four points $P_1,\, P_2,\, P_3$ and $P_4$.\par

By the definition of a good arrival grid, and since $\gamma$ is affine around $P_1$, there exist a small interval $(t_1^-,t_1^+)$, containing $\gamma^{-1}(P_1)$, such that $\u P_1^-:=u\circ \gamma(t_1^-)$ is outside the rectangle $\u\QQ_{n,m}$, while $\u P_1^+ =u\circ\gamma(t_1^+)$ is in its interior. Similarly, we define the intervals $(t_2^-,t_2^+),\, (t_3^-,t_3^+)$ and $(t_4^-,t_4^+)$, respectively containing $\gamma^{-1}(P_2),\, \gamma^{-1}(P_3)$ and $\gamma^{-1}(P_4)$, and with points $\u P_2^-,\, \u P_3^-$ and $\u P_4^+$ in the interior of $\QQ_{m,n}$ and $\u P_2^+,\, \u P_3^+$ and $\u P_4^-$ outside of it. We can select the four intervals so small that they are disjoint, and that the points $\u P_i^\pm$ are arbitrarily close to the corresponding $\u P_i$, in particular with a distance much smaller than $\min_{1\leq i\neq j\leq 4} |\u P_i-\u P_j|$.\par

Since $u$ is NCL, we can find a continuous and injective map $\varphi:[0,1]\to\QQ$ such that $\delta:=\| \varphi-u\circ \gamma\|_{L^\infty([0,1])}$ is arbitrarily small. We can take $\delta$ so small that the points $\varphi(t_1^-),\,\varphi(t_2^+),\,\varphi(t_3^+)$ and $\varphi(t_4^-)$ are outside of $\u\QQ_{n,m}$, while the points $\varphi(t_1^+),\,\varphi(t_2^-),\,\varphi(t_3^-)$ and $\varphi(t_4^+)$ are in its interior. Since $\varphi$ is continuous, we can define $t_1$ (resp., $t_4$) as the last point of the interval $(t_1^-,t_1^+)$ (resp., $(t_4^-,t_4^+)$) such that $\varphi(t_1)$ (resp., $\varphi(t_4)$) belongs to $\partial\u\QQ_{n,m}$. Similarly, $t_2$ and $t_3$ are the first points of the intervals $(t_2^-,t_2^+)$ and $(t_3^-,t_3^+)$ such that $\varphi(t_2)$ and $\varphi(t_3)$ belong to $\partial\u\QQ_{n,m}$. Calling $\widetilde{\u P}_i=\varphi(t_i)$ for $i=1,\,2,\,3,\,4$, by the smallness of $\delta$ we have that the points $\widetilde{\u P}_1,\,\widetilde{\u P}_2,\,\widetilde{\u P}_3$ and $\widetilde{\u P}_4$ are arbitrarily close to the corresponding points $\u P_1,\,\u P_2,\,\u P_3$ and $\u P_4$, hence in particular the order on $\partial\u\QQ_{n,m}$ is the same.\par

Keep now in mind that $u\circ \gamma$ is contained in the interior of $\u\QQ_{n,m}$ in the intervals $[t_1^+,t_3^-]$ and $[t_4^+,t_2^-]$. As a consequence, again provided that $\delta$ is small enough, we deduce that $\varphi$ is contained in the interior of $\u\QQ_{n,m}$ in the open segments $(t_1,t_3)$ and $(t_2,t_4)$. Since $\varphi$ is injective, this yields that $\widetilde{\u P}_1$ and $\widetilde{\u P}_3$ divide $\partial\QQ_{n,m}$ in two parts, and both $\widetilde{\u P}_2$ and $\widetilde{\u P}_4$ are on the same one of the two. As noticed above, the same is true with the points $\u P_i$ instead of the points $\widetilde{\u P}_i$, hence the claim is proved.\par

Thanks to Claim~1, there is exactly one intersection point between the generalised segments $[\u P_1,\u P_2]$ and $[\u P_3,\u P_4]$. Since these two generalised segments have to be the images under $v$ of the segments $P_1P_2$ and $P_3P_4$ respectively, we must define $\u V=v(V)=[\u P_1,\u P_2]\cap[\u P_3,\u P_4]$.\par

Summarizing, we have defined the image $\u P_i$ of each of the points $P_i$, as well as the image $\u V_{p,q}$ of each vertex $V_{p,q}$. Notice that $\G$ is a finite union of essentially disjoint segments, each of which has both endpoints and no internal point in the set $\{P_i,\, 1\leq i \leq H\}\cup \{V_{p,q},\, 1\leq p,\,q\leq K-1\}$. We are finally in position to give the definition of $v$. Calling $AB$ the generic segment of the above form, on the segment $AB$ we let $v$ be the path $[\u A,\, \u B]$, parametrised at constant speed, which makes sense since the points $\u A$ and $\u B$ belong to a same rectangle $\u\QQ_{n,m}$ by construction.\par

Notice that, as decided above, the image under $v$ of every simple segment $P_iP_j$ is the generalised segment $[\u P_i,\u P_j]$. Moreover, the parametrisation of $v$ has constant speed if the simple segment does not contain any vertex, while otherwise the velocity might change from a part of the segment to another, and in principle $v$ might even not be injective on a simple segment. Since by construction it is clear that $v$ is continuous and equals the identity on $\partial\QQ$, to conclude the proof we only have to check that $v$ is injective.\par

It is simple to observe that, in order to establish the injectivity of $v$, it is enough to show that the images of any two disjoint simple segments have empty intersection. Up to renumbering, let us then assume that $P_1P_2$ and $P_3P_4$ are two simple segments with $P_1P_2\cap P_3P_4=\emptyset$. We have to show that
\begin{equation}\label{tte}
[\u P_1,\u P_2]\cap [\u P_3,\u P_4]=\emptyset\,.
\end{equation}
Since the points $\u P_i$ are all distinct, and the interiors of the generalised segments are in the interiors of the corresponding rectangles of $\widetilde\G$, (\ref{tte}) is obvious if the simple segments $P_1P_2$ and $P_3P_4$ are not associated with the same rectangle. Let us then assume that both $[\u P_1,\u P_2]$ and $[\u P_3,\u P_4]$ belong to a same rectangle $\u\QQ_{n,m}$. To show~(\ref{tte}), we have to check that the points $\u P_1,\, \u P_2,\, \u P_3$ and $\u P_4$ have the correct order in $\partial\u\QQ_{n,m}$, namely, $\u P_3$ and $\u P_4$ are on the same one of the two connected components in which $\partial\u\QQ_{n,m}$ is divided by $\u P_1$ and $\u P_2$. And finally, to obtain this fact one has to argue exactly as in the proof of Claim~1, considering a Lipschitz, injective, admissible curve which contains both the simple segments $P_1P_2$ and $P_3P_4$.
\end{proof}


\begin{thebibliography}{XX}
\bibitem{HP} S. Hencl \& A. Pratelli, Diffeomorphic Approximation of $W^{1,1}$ Planar Sobolev Homeomorphisms, J. Eur. Math. Soc., {\bf 20} (2018), no. 3, 597--656.
\bibitem{DPP} G. De Philippis \& A. Pratelli, The closure of planar diffeomorphisms in Sobolev spaces, to appear on Ann. Inst. H. Poincar\'e Anal. Non Lin\'eaire (2019).
\bibitem{MS} S. M\"uller \& S.J. Spector, An existence theory for nonlinear elasticity that allows for cavitation, Arch. Rat. Mech. Anal. {\bf 131} (1995), 1--66.
\bibitem{PR} A. Pratelli, \& E. Radici, Approximation of planar BV homeomorphisms by diffeomorphisms, J. Func. Anal. {\bf 276} (2019), 659--686.
\bibitem{Sve} V. \v Sver\'ak, Regularity properties of deformations with finite energy. Arch. Rational Mech. Anal. {\bf 100} (1988), 105--127.
\end{thebibliography}
\end{document}